\documentclass{amsart}

\newcommand{\noopsort}[1]{}

\usepackage{amsthm,amssymb}
\usepackage{tabularx}
\usepackage{color,caption}
\usepackage{amsrefs}
\usepackage[algoruled,vlined,english,linesnumbered]{algorithm2e}
\usepackage{graphicx}
\usepackage{pgfplots}
\usepackage{longtable}
\usepackage{hyperref}

\DeclareMathOperator{\Gal}{Gal}
\DeclareMathOperator{\Aut}{Aut}
\DeclareMathOperator{\Stab}{Stab}
\DeclareMathOperator{\ggcd}{gcd}

\newcommand{\ZZ}{\mathbb{Z}}
\newcommand{\FF}{\mathbb{F}}
\newcommand{\Fp}{\mathbb{F}_p}
\newcommand{\Zp}{\ZZ_p}
\newcommand{\QQ}{\mathbb{Q}}
\newcommand{\Qp}{\QQ_p}
\newcommand{\bQp}{\bar{\QQ}_p}
\newcommand{\GQp}{\mathcal{G}}
\newcommand{\hatZ}{\hat{\ZZ}}

\newcommand{\sss}{\mathrm{ss}}
\newcommand{\xc}{\mathrm{xc}}

\newtheorem{theorem}{Theorem}[section]

\newtheorem{proposition}[theorem]{Proposition}
\newtheorem{corollary}[theorem]{Corollary}
\theoremstyle{definition}
\newtheorem{remark}[theorem]{Remark}
\newtheorem{definition}[theorem]{Definition}
\newtheorem{conjecture}[theorem]{Conjecture}

\title{The Inverse Galois Problem for $p$-adic fields}
\author{David Roe}
\thanks{Supported by Simons Foundation grant 550033}
\address{Department of Mathematics, MIT 2-106, 77 Massachusetts Ave, Cambridge, MA 02139}
\email{roed@mit.edu}

\begin{document}
\begin{abstract}
We describe a method for counting the number of extensions of $\Qp$ with a given Galois group $G$, founded upon the description of the absolute Galois group of $\Qp$ due to Jannsen and Wingberg.
Because this description is only known for odd $p$, our results do not apply to $\QQ_2$.  We report on the results of counting such extensions for $G$ of order up to $2000$ (except those divisible by $512$),
for $p=3,5,7,11,13$.  In particular, we highlight a relatively short list of minimal $G$ that do not arise as Galois groups.  Motivated by this list, we prove two theorems about the inverse Galois problem for $\Qp$:
one giving a necessary condition for $G$ to be realizable over $\Qp$ and the other giving a sufficient condition.
\end{abstract}
\maketitle

\section{Introduction}

The inverse Galois problem is most commonly studied over $\QQ$.  There, a theorem of Shafarevich \citelist{\cite{shafarevich:54a}\cite{neukirch-schmidt-wingberg:CohomologyOfNumberFields}*{Thm. 9.6.1}}
shows that every solvable group is realizable as the Galois group of an extension of $\QQ$.  Attention has thus focused on simple groups, and many have been shown to be realizable; see \cite{malle-matzat:InverseGaloisTheory} for background.

Over $\QQ$, if a given group arises as a Galois group it will arise for infinitely many extensions.  Thus the constructive version of the problem, finding extensions with a given Galois group,
has been approached by the method of generic polynomials.  A generic polynomial for a group $G$ is a monic polynomial with coefficients in a function field $\QQ(c_1, \dots, c_n)$
so that every extension of $\QQ$ with Galois group $G$ will arise via specializing the $c_i$ to elements of $\QQ$.  Even if $G$ is realizable, it may not have a generic polynomial parameterizing all extensions.

Over $\Qp$, for fixed $p$ and $G$, there are only finitely many isomorphism classes of Galois extensions $K/\Qp$ with $\Gal(K/\Qp) \cong G$.  Thus, rather than trying to produce them via a
generic polynomial, one could hope to enumerate them directly.  As a first step toward such an enumeration, in this paper we study the less refined question of counting such $K$.

The counting and enumeration of $p$-adic fields has a rich history, mostly separate from the study of the inverse Galois problem.  Rather than focusing on the Galois group, most approaches have
studied the extensions of a given degree, or with a given degree and discriminant.  Foundational work of Krasner \cite{krasner:66a}*{Thm. 2} gave counts for the number of extensions of degree $n$
in a fixed algebraic closure, and Serre \cite{serre:78a} gives a ``mass formula'' where the counts are weighted appropriately.  More recently, Hou and Keating \cite{hou-keating:04a} and
Monge \cite{monge:11a} have described how to count isomorphism classes of extensions with prescribed ramification and inertia degrees.

There has been some work on counting extensions with a given Galois group.  When $G$ is a $p$-group generated by $d$ elements (minimally) and $k/\Qp$ has degree $n$,
Shafarevich \cite{shafarevich:47a} has obtained the following formula for the number of extensions of $k$ with Galois group $G$, using his description of the maximal pro-$p$ quotient of the absolute Galois group:
\begin{equation} \label{eq:shaf}
\frac{1}{\lvert \Aut(G) \rvert} \left(\frac{\lvert G \rvert}{p^d}\right)^{n+1} \prod_{i=0}^{d-1} (p^{n+1} - p^i).
\end{equation}
The result only holds for $k$ that do not contain the $p$-th roots of unity, but Yamagishi \cite{yamagishi:95a} has generalized it, obtaining a formula involving characters of $G$.

Other authors have pursued the problem of enumerating $p$-adic fields \cites{pauli-roblot:01a,jones-roberts:06a} of a given degree.  Theoretically, this would solve the problem of
enumerating with a given Galois group, since one can determine from $G$ the smallest degree where a field can have a normal closure with Galois group $G$. 
However, for many groups this degree is prohibitively large for the methods employed, since you also get many other, much larger, Galois groups at the same time.

In this paper, we count Galois extensions with Galois group $G$ by exploiting the explicit description of the absolute Galois group of $\Qp$.  This approach has the benefit of completely avoiding computations with polynomials,
allowing for a large number of groups to be considered.  The downside is that we do not get any information on many invariants of number theoretic interest, such as the discriminant or the ramification filtration,
beyond distinguishing between tame and wild inertia.

We have chosen to focus on the case of $\Qp$ because it has the most intrinsic interest, and because the number of extensions grows exponentially with the absolute degree of the base field, as illustrated by \eqref{eq:shaf}.
The code, which uses GAP \cite{gap} and SageMath \cite{sage}, can be found at \url{https://github.com/roed314/padicIGP}.

\subsection{Summary}

We begin Section \ref{sec:pot} with the notion of a \emph{potentially $p$-realizable group}, which encapsulates the obvious conditions on $G$ that come from the first few steps of the ramification filtration.
This notion is closed under quotients, and we conjecture that any potentially $p$-realizable group can be expressed as a semidirect product of its $p$-core and its tame quotient.
This conjecture is supported by experimental evidence, and has consequences for the existence of subextensions complementary to the maximal tame subextension.
We close with Section \ref{ssec:enum_pot}, where we give algorithms to test whether a group is potentially $p$-realizable and to enumerate such groups.

In Section \ref{sec:absgal} we review the structure of the absolute Galois group, which plays a key role in our approach to counting extensions.  We use the description to show that our notion of potentially $p$-realizable
has the property that any such group will be realized over some $p$-adic field $k$.

Section \ref{sec:counting} describes the algorithms used to count extensions $K/\Qp$ with a given Galois group $G$.  We give an explicit enumeration in the case of abelian groups, since we need this as a base case
for inductive lifting methods later.  We then summarize the tame case, which follows from the well-known structure of the tame quotient of $\Gal(\bar{\QQ}_p/\Qp)$.  Finally we give a lifting method for counting
extensions for arbitrary $G$, and briefly discuss its runtime.

In Section \ref{sec:invprob} we apply the counting algorithms to the question of whether a potentially $p$-realizable group is actually realized over $\Qp$.  We start by listing minimal examples of groups that are unrealizable.
We then proceed, in Section \ref{ssec:realizable}, to prove Theorems \ref{thm:multiplicity} and \ref{thm:converse} giving one necessary and one sufficient condition for $p$-realizability.
Both conditions relate to the structure of the $p$-core of $G$ as a representation of the tame quotient.

\subsection{Notation and Terminology}

We work throughout with a prime $p \ne 2$ and a finite group $G$.  There are some naturally defined subgroups of $G$ that will play an important role throughout the paper.
The $p$-core $V$ of $G$ is the intersection of all of the $p$-Sylow subgroups of $G$:
\[
V = \bigcap_{P \text{ $p$-Sylow}} P.
\]
It is the maximal normal $p$-group inside $G$.
The quotient $T = G/V$ has the structure of a metacyclic group
(an extension of a cyclic group by a cyclic group), but not canonically.
 It acts on $V$ by conjugation.  We call $G$ \emph{tame} if $V$ is trivial and $G=T$.
 
We will also use the Frattini subgroup $W$ of $V$, defined as
\[
W = V^p V'
\]
where $V'$ is the commutator subgroup of $V$.  The quotient $V/W$ is the maximal quotient of $V$ that is an elementary abelian $p$-group.  The action of $T$ on $V$ descends to an action on $V/W$, yielding a representation of $T$ on an $\Fp$-vector space.

We will refer to groups by their ID in GAP's SmallGroups library \cite{smallgroups} using the notation $n$G$k$, where $n$ is the order of $G$ and $k$ enumerates groups of that order.

Write $\GQp$ for the absolute Galois group $\Gal(\bar{\QQ}_p/\Qp)$.

\section{Potentially $p$-realizable groups} \label{sec:pot}

\subsection{The structure of $p$-adic Galois groups}

The structure of $p$-adic field extensions \cite{hasse:NumberTheory}*{Chapter 16} imposes constraints on the possible Galois groups that can arise.
Any finite extension $K \supseteq \Qp$ can be decomposed into a tower $K \supseteq K_t \supseteq K_u \supseteq \Qp$, where $K_u/\Qp$ is unramified, $K_t/K_u$ is tame and totally ramified,
and $K/K_t$ is totally wildly ramified.   When $K/\Qp$ is Galois, this tower corresponds to the first parts of the ramification filtration on $G = \Gal(K/\Qp)$:
\begin{equation} \label{eq:Gfilt}
G = G_{-1} \supseteq G_0 \supseteq G_1.
\end{equation}
The fixed field of $G_0$ is the unramified subfield $K_u$ and the quotient $G/G_0$ must be cyclic.  The fixed field of $G_1$ is the tame subfield $K_t$ and the quotient $G_0/G_1$
must be cyclic of order relatively prime to $p$.  Finally, $G_1 \cong \Gal(K/K_t)$ is a $p$-group.  Moreover, $G_0$ and $G_1$ are normal subgroups of $G$.

By a theorem of Iwasawa \cite{iwasawa:55a}*{Thm. 2}, the Frobenius element of $G/G_0$ acts on $G_0/G_1$ by raising to the $p$th power.

\begin{definition}
A group $G$ is \emph{potentially $p$-realizable} if it has a filtration $G \supseteq G_0 \supseteq G_1$ so that
\begin{enumerate}
\item $G_0$ and $G_1$ are normal in $G$,
\item $G/G_0$ is cyclic, generated by some $\sigma \in G$,
\item $G_0/G_1$ is cyclic of order relatively prime to $p$, generated by some $\tau \in G_0$,
\item $\tau^\sigma = \tau^p$,
\item $G_1$ is a $p$-group.
\end{enumerate}
We will call such a filtration on $G$ a \emph{tame structure}.
A group $G$ is \emph{$p$-realizable} if there exists an extension $K/\Qp$ with $\Gal(K/\Qp) \cong G$.
\end{definition}

\begin{remark}
By the discussion above, any $p$-realizable group is potentially $p$-realizable, justifying the terminology.
We will also see in Proposition \ref{prop:potpreal} that if $G$ is potentially $p$-realizable then it arises
as a Galois group after some finite extension, conforming with the common usage of ``potentially.''
\end{remark}

\begin{remark}
Since every $p$-group is nilpotent, the condition that $G$ is potentially $p$-realizable implies that $G$ is solvable.  However, some groups $G$ may have multiple tame structures.
The simplest example is $G = C_2$ and $p$ odd, where we can take $G_0 = G$ or $G_0 = 1$.  An example with varying $G_1$ is $G = C_{p^2}$, where we can take $G_0 = G_1 = C_p$ or $G_0 = G_1 = 1$.
\end{remark}
\begin{proposition}
Any quotient of a potentially $p$-realizable group is potentially $p$-realizable.
\end{proposition}
\begin{proof}
Suppose $G$ has tame structure $G \supseteq G_0 \supseteq G_1$ and $N \trianglelefteq G$.  It suffices to show that $G/N \supseteq G_0 N / N \supseteq G_1 N / N$ is a tame structure on $G/N$.

By the third isomorphism theorem, $(G/N) / (G_0N/N) \cong G / (G_0 N)$ is a quotient of $G / G_0$ and thus cyclic, generated by the image of $\sigma$.
The natural map $G_0 \to (G_0 N / N) / (G_1 N / N) \cong (G_0 N) / (G_1 N) \cong G_0 / (G_1(G_0 \cap N))$ has kernel containing $G_1$, showing that $(G_0 N / N) / (G_1 N / N)$ is cyclic and generated by the image of $\tau$.

Since the relation $\tau^\sigma = \tau^p$ holds in $G$, it also holds for the images of $\sigma$ and $\tau$ in $G/N$.  Finally, $G_1 N / N \cong G_1 / (G_1 \cap N)$ is a $p$-group since $G_1$ is.
\end{proof}

If $G$ is potentially realizable, the maximal choice for $G_1$ is the $p$-core $V$.
We may always enlarge a tame structure on $G$ to make $G_1 = V$:

\begin{proposition} \label{prop:Vsuffices}
If $G \supseteq G_0 \supseteq G_1$ is a tame structure on $G$, so is $G \supseteq G_0 V \supseteq V$.
\end{proposition}
\begin{proof}
Since $G_0$ and $V$ are normal subgroups of $G$, so is $G_0 V$.  Moreover, $G / (G_0 V)$ is a quotient of $G/G_0$ and thus cyclic generated by the same $\sigma \in G$.
Since the order of $G_0/G_1$ is prime to $p$, $G_0 \cap V = G_1$ and the second isomorphism theorem implies that $(G_0 V) / V \cong G_0 / G_1$ with the image of $\tau$ still generating $(G_0 V) / V$.
\end{proof}

Define $T = G / V$, the smallest possible tame quotient of $G$.

\begin{conjecture} \label{conj:semidirect}
If $G$ is potentially $p$-realizable, then $G \cong V \rtimes T$.
\end{conjecture}

The conjecture holds for $p \in \{3, 5, 7, 11, 13\}$ and potentially $p$-realizable groups $G$ with $\lvert G \rvert \le 2000$.
It also holds when $T$ has order prime to $p$, by the Schur-Zassenhaus theorem.
Note that we may not replace $V$ with an arbitrary $G_1$, as the example of $C_{p^2} \supseteq C_p \supseteq C_p$ shows.
Moreover, attempting to decompose the pieces further fails.  The tame quotient $T$ is not necessarily the semidirect product of $G_0/G_1$ by $G/G_0$:
the quaternion group of order $8$ is $p$-realizable for $p \equiv 3 \pmod{4}$ but not a semidirect product of cyclic subgroups.

The conjecture has an interesting corollary for $p$-adic fields.
\begin{corollary}
Assume Conjecture \ref{conj:semidirect} holds, and suppose that $K/\Qp$ is Galois.
If $K_t/\Qp$ is the maximal tamely ramified subextension of $K/\Qp$ and $\Gal(K/K_t)$ is the $p$-core of $\Gal(K/\Qp)$ then there is a totally wildly ramified complement $K_0/\Qp$ with $K = K_0 K_t$.
\end{corollary}

\subsection{Enumerating small examples} \label{ssec:enum_pot}

The first step toward counting $p$-adic fields by Galois group is computing a list of potential $G$.
Since GAP's database of small groups \cite{smallgroups} can identify groups of order $n$ for $n \le 2000$ except
$n = 512, 1024, 1536$, groups of these orders were screened.

When $n$ is prime to $p$, we may use the classification of metacyclic groups \cite{hempel:00a}*{Lemma 2.1} to screen $G$.  This process is described in Algorithm \ref{alg:tame}.

\begin{algorithm}
\SetKw{KwFrom}{from}
\SetKw{KwBy}{by}
\SetKwData{id}{id}\SetKwData{step}{step}\SetKwData{groups}{groups}
\SetKwFunction{sorted}{sorted}\SetKwFunction{gcd}{gcd}
\SetKwInOut{Input}{Input}\SetKwInOut{Output}{Output}
\Input{An integer $n$}
\Output{The list of potentially $p$-realizable groups of order $n$ with trivial $G_1$}
\BlankLine
\groups = []\;
\For{positive $k, m$ with $n = k \cdot m$}{
\If{$m$ divides $p^k - 1$}{
\step = $m \  / $ \gcd{$m$, $p - 1$}\;
\For{$\ell$ \KwFrom $0$ \KwTo $m$ \KwBy \step}{
Find the GAP \id of $\langle x, y \;\vert\; x^k = y^\ell, y^m = 1, y^x = y^p \rangle$\;
Add \id to \groups if not present\;
}}}
\Return{\sorted{\groups}}\;
\caption{Finding potentially $p$-realizable groups: the tame case} \label{alg:tame}
\end{algorithm}

When $n$ has $p$-adic valuation $1$, we can build groups as extensions of metacyclic groups.
Any group of order $n$ will arise either as an extension of a group of order $n/p$ by $C_p$, or
as a metacyclic group produced by Algorithm \ref{alg:tame}.  The extensions are computable using
GAP's \texttt{Extensions} method, and we describe the process in Algorithm \ref{alg:ext}.

\begin{algorithm}
\SetKwData{groups}{groups}\SetKwData{T}{T}\SetKwData{G}{G}\SetKwData{id}{id}
\SetKwFunction{sorted}{sorted}\SetKwFunction{Extensions}{Extensions}
\SetKwInOut{Input}{Input}\SetKwInOut{Output}{Output}
\Input{An integer $n$ with $v_p(n) = 1$}
\Output{The list of potentially $p$-realizable groups of order $n$}
\BlankLine
\groups = []\;
\ForEach{tame group \T of order $n/p$}{
\ForEach{homomorphism $\phi$ from \T to $\Aut(C_p)$}{
\ForEach{group \G in \Extensions{\T, $\phi$}}{
\If{$x$ and $y$ lift to elements of $G$ satisfying the tame relation}{
Find the GAP \id of \G\;
Add \id to \groups if not present\;
}}}}
\ForEach{tame group \T of order $n$}{
Find the \id of \T\;
Add \id to \groups if not present\;
}
\Return{\sorted{\groups}}\;
\caption{Finding potentially $p$-realizable groups: valuation $1$} \label{alg:ext}
\end{algorithm}

When $n$ has larger $p$-adic valuation, this extension method becomes more complicated,
since there are more possibilities for $V$.  Moreover, some of the possible $V$
are not elementary abelian $p$-groups, so GAP's \texttt{Extensions} method does not apply.
While it would be possible to try to construct the extensions manually using
GAP's \texttt{GrpConst} package \cite{grpconst}, in practice it suffices to check whether each
group in the small group database \cite{smallgroups} with order $n$ is potentially $p$-realizable
using Algorithm~\ref{alg:preal}.

\begin{algorithm}
\SetKw{continue}{continue}
\SetKwData{N}{N}\SetKwData{D}{D}\SetKwData{G}{G}\SetKwData{V}{V}\SetKwData{T}{T}\SetKwData{True}{True}\SetKwData{False}{False}
\SetKwFunction{DerivedSubgroup}{DerivedSubgroup}
\SetKwFunction{IsCyclic}{IsCyclic}
\SetKwFunction{PCore}{PCore}
\SetKwFunction{NormalSubgroupsContaining}{NormalSubgroupsContaining}
\SetKwInOut{Input}{Input}\SetKwInOut{Output}{Output}
\Input{A group \G}
\Output{Whether or not \G is potentially $p$-realizable.}
\BlankLine
\V = \PCore{\G}\;
\T = \G / \V\;
\If{\IsCyclic{\T}}{
\Return{\True}\;
}
\D = \DerivedSubgroup{\G}\;
\If{\IsCyclic{\D}}{
\For{\N in \NormalSubgroupsContaining{\D}}{
\If{\IsCyclic{\N} and \IsCyclic{\G/\N}}{
Let $e$ be the order of \N and $f$ the order of \G/\N\;
Let $a$ be the exponent in the conjugation action of \G/\N on \N\;
Find $b$ with $a^b \equiv p \pmod{e}$, or \continue if not possible\;
Let $m$ be the order of $a \pmod{e}$\;
\If{$\gcd(m, b, f) = 1$}{
\Return{\True}\;
}}}}
\Return{\False}\;
\caption{Determining whether a group is potentially $p$-realizable} \label{alg:preal}
\end{algorithm}

\section{The absolute Galois group of a local field} \label{sec:absgal}

Our approach to counting $p$-adic fields rests on the following description of the absolute Galois group of $\Qp$.  Let $p \ne 2$, $k$ be a $p$-adic field, $N = [k : \Qp]$, $q$ the cardinality of the residue field of $k$,
and $p^s$ the order of the group of $p$-power roots of unity in the maximal tame extension $k^t/k$.  Choose $g, h \in \Zp$ with
\[
\zeta^\sigma = \zeta^g, \zeta^\tau = \zeta^h \text{ for } \zeta \in \mu_{tr},
\]
where $\sigma, \tau \in \Gal(k^t/k)$ with $\tau^\sigma = \tau^q$ as in \cite{iwasawa:55a}, and $\mu_{tr}$ the $p$-power roots of unity in $k^t$.

Let $\pi = \pi_p$ be the element of $\hatZ = \prod_\ell \ZZ_\ell$ with coordinate $1$ in the $\Zp$-component and $0$ in the $\ZZ_\ell$ components for $\ell \ne p$.  Then for $x, y$ in a profinite group\footnote{See \cite{ribes-zalesskii:profinite_groups}, especially sections 3.3 and 4.1, for relevant background on profinite groups.}, set
\[
\langle x, y \rangle = (x^{h^{p-1}}y x^{h^{p-2}} y \cdots x^h y)^{\frac{\pi}{p-1}}.
\]

\begin{theorem}[{\cite{neukirch-schmidt-wingberg:CohomologyOfNumberFields}*{Thm. 7.5.14}}] \label{thm:GQp_desc}
The absolute Galois group $\Gal(\bar{k}/k)$ is isomorphic to the profinite group generated by $N+3$ generators $\sigma, \tau, x_0, \dots, x_N$, subject to the following conditions and relations.
\begin{enumerate}
\item The closed subgroup topologically generated by $x_0, \dots, x_N$ is normal in $G$ and is a pro-$p$-group.
\item The elements $\sigma, \tau$ satisfy the tame relation
\[
\tau^\sigma = \tau^q.
\]
\item The generators satisfy the following wild relation.  If $N$ is even then
\[
x_0^\sigma = \langle x_0, \tau \rangle^g x_1^{p^s} [x_1, x_2] [x_3, x_4] \dots [x_{N-1}, x_N].
\]
If $N$ is odd then 
\[
x_0^\sigma = \langle x_0, \tau \rangle^g x_1^{p^s} [x_1, y_1] [x_2, x_3] \dots [x_{N-1}, x_N],
\]
where $g$ and $s$ are defined above and $y_1$ is an explicit element in the span of $x_1, \sigma,$ and $\tau$, specified below when $k=\Qp$.  
\end{enumerate}
\end{theorem}

We will mostly be interested in the case where $k = \Qp$; recall that we write $\GQp$ for $\Gal(\bar{\QQ}_p/\Qp)$.
Now $q=p$, $g=1$ and $h$ is a $(p-1)$st root of unity in $\Zp$.
In order to define $y_1$, let $\Qp^t$ be the maximal tamely ramified extension of $\Qp$ and define
$\beta : \Gal(\Qp^t/\Qp) \to \Zp^\times$ by setting $\beta(\sigma) = 1$ and $\beta(\tau) = h$.
For $\rho$ in the subgroup of $\GQp$ generated by $\sigma$ and $\tau$ and $x \in \GQp$, set
\[
\{x, \rho\} = (x\rho^2x^{\beta(\rho)}\rho^2\dots x^{\beta(\rho^{p-2})}\rho^2)^{\frac{\pi}{p-1}}.
\]
Let $\pi_2 \in \hatZ$ be the element with $\pi_2 \hatZ = \ZZ_2$, and set $\tau_2 = \tau^{\pi_2}$ and $\sigma_2 = \sigma^{\pi_2}$.
Set
\begin{equation} \label{eq:y1def}
y_1 = x_1^{\tau_2^{p+1}}\{x_1, \tau_2^{p+1}\}^{\sigma_2\tau_2^{(p-1)/2}}\{\{x_1, \tau_2^{p+1}\}, \sigma_2\tau_2^{(p-1)/2}\}^{\sigma_2\tau_2^{(p+1)/2}+\tau_2^{(p+1)/2}}.
\end{equation}
The wild relation for $\Qp$ then becomes
\begin{equation} \label{eq:wild_rel}
x_0^\sigma = \langle x_0, \tau \rangle x_1^p [x_1, y_1]
\end{equation}

We can use this description of absolute Galois groups to show that any potentially $p$-realizable group occurs as a Galois group over some $k$ with $k/\Qp$ finite.

\begin{proposition} \label{prop:potpreal}
If $G$ is potentially $p$-realizable and $V/W$ has dimension $m$ then $G$ will be realized over $k$ if $[k:\Qp] \ge 2m+1$.
\end{proposition}
\begin{proof}
It suffices to exhibit a surjective homomorphism $\Gal(\bar{k}/k) \to G$, which we define by specifying the images of the generators.
Map $x_0, x_1, x_3, x_5, \dots, x_{2m+1}$ and $x_{2m+2}, \dots, x_N$ to $1$.  Then the wild relation is automatically satisfied, and we may freely choose the images of $x_2, \dots, x_{2m}$.
As long as we map them to elements of $V$ that project to an $\Fp$-basis of $V/W$, Burnside's basis theorem implies that they will generate $V$.  The fact that $G$ is potentially $p$-realizable
then implies that we may extend this homomorphism to a surjective map on all of $\Gal(\bar{k}/k)$.
\end{proof}
Note that one can decrease $2m+1$ in some cases using the representation of $T$ on $V$, and even then this bound is certainly not sharp.

\section{Counting $p$-adic fields} \label{sec:counting}

\subsection{Parameterizing extensions}
Following \cite{yamagishi:95a}, we count the extensions of $\Qp$ with Galois group $G$ by counting the surjections $\GQp \to G$, modulo automorphisms of $G$.
We can then translate the description of $\GQp$ from Theorem \ref{thm:GQp_desc} to a counting problem in $G$.  Let $n$ be the order of $G$ and factor $n = u_p p^r = u_2 2^s$ with $(u_p,p) = 1$ and $u_2$ odd.
Using the Chinese remainder theorem, define integers $a$ and $b$ so that
\begin{align*}
a &= 0 \pmod{u_p} & (p-1)a &= 1 \pmod{p^r} \\
b &= 0 \pmod{u_2} & b &= 1 \pmod{2^s}.
\end{align*}

Since the images of $x_0$ and $x_1$ have $p$-power order, they lie in $V$.% is the maximal normal $p$-group in $G$, it contains the images of $x_0$ and $x_1$.  
\begin{definition} \label{def:TGXG}
Define $T_G$ to be the set of pairs $(\sigma, \tau) \in G^2$ so that
\begin{enumerate}
\item $\tau^{\sigma} = \tau^p$,
\item the images of $\sigma$ and $\tau$ in $G/V$ generate $G/V$.
\end{enumerate}

Define $X_G$ to be the set of quadruples $(\sigma, \tau, x_0, x_1) \in G^4$ so that
\begin{enumerate}
\item $\tau^{\sigma} = \tau^p$,
\item $x_0, x_1 \in V$,
\item $\sigma, \tau, x_0, x_1$ generate $G$,
\item $x_0^{\sigma} = \langle x_0, \tau \rangle x_1^p [x_1, y_1]$,
\end{enumerate}
where $y_1$ is defined as in \eqref{eq:y1def}.
\end{definition}
Note that we may compute the projections $\pi/(p-1)$ and $\pi_2$ by raising to the $a$ and $b$ powers respectively.

\begin{proposition}
The Galois extensions of $\Qp$ with Galois group $G$ are in bijection with the orbits of $X_G$ under the action of $\Aut(G)$.
\end{proposition}
\begin{proof}
Finite extensions $K$ of $\Qp$ within a fixed algebraic closure of $\Qp$ correspond to finite index subgroups $H_K$ of $\GQp$.  The condition that $K$ is Galois with Galois group $G$ translates
to the condition that $H_K$ is normal with $\GQp/H_K \cong G$.  Different subgroups $H$ cannot yield isomorphic $K$ since an isomorphism of fields would extend to an automorphism of $\bQp$
conjugating one $H$ to the other, which is impossible since both are normal.  Finally, elements of $X_G$ correspond to homomorphisms $\GQp \to G$ by the description of $\GQp$ in Theorem \ref{thm:GQp_desc},
and the kernel of such a homomorphism is preserved by composition with an automorphism of $G$.
\end{proof}

We will be inductively constructing representatives for the orbits of $\Aut(G)$ on $X_G$; write $Y_G$ for a choice of such representatives.  Then $Y_G$ will be in bijection with the extensions of $\Qp$ with Galois group $G$.

\subsection{Abelian groups} \label{ssec:ab}

When $G$ is abelian, the wild relation simplifies to $x_0 = x_1^p$.  Thus $x_0$ is determined by $x_1$, and the wild relation imposes no constraint on $x_1$.
The order of $\tau$ must divide $p-1$, the order of $x_1$ must be a power of $p$, and the three elements $\sigma, \tau$, and $x_1$ must generate $G$.

Write
\begin{equation} \label{eq:Gdef}
G \cong \prod_{\ell} \prod_{i = 1}^{m_\ell} \ZZ / \ell^{n_{\ell, i}} \ZZ,
\end{equation}
where $n_{\ell, 1} \le \dots \le n_{\ell, m_\ell}$ for each $\ell$.  We can enumerate the elements of $X_G$ as a function of the $n_{\ell, i}$.  Let $\alpha_\ell$ be the element of $G$ with a $1$ in the $\ell, 1$ component
and $0$s elsewhere, and let $\beta_\ell$ be the element with a $1$ in the $\ell, 2$ component and $0$s elsewhere.  Since we will be analyzing the $\ell$-components separately, we drop $\ell$ from the notation, writing $a$ for $n_{\ell, 1}$, $b$ for $n_{\ell, 2}$, $\alpha$ for $\alpha_\ell$ and $\beta$ for $\beta_\ell$.
\begin{enumerate}
\item In the case $m_\ell \ge 3$, set $c_\ell = 0$ and $C_\ell = \{\}$.
\item In the case $m_\ell = 2$, if $a \ne b$ and $\ell=p$, set $c_\ell = 2$ and \\ $C_\ell = \{(\alpha, 0, p\beta, \beta), (\beta, 0, p\alpha, \alpha)\}$.
\item \label{case3} In the case $m_\ell = 2$, if $a \ne b$ and $\ell^b$ divides $p-1$, set $c_\ell = 2$ and \\ $C_\ell = \{(\alpha, \beta, 0, 0), (\beta, \alpha, 0, 0)\}$.
\item In the case $m_\ell = 2$, if $a = b$ and $\ell=p$, set $c_\ell = 1$ and $C_\ell = \{(\alpha, 0, p\beta, \beta)\}$.
\item In the case $m_\ell = 2$, if $\ell^a$ divides $p-1$ but case (\ref{case3}) does not apply, set $c_\ell = 1$ and $C_\ell = \{(\beta, \alpha, 0, 0)\}$.
\item In the case $m_\ell = 2$, if $\ell \ne p$ and $\ell^a \nmid p-1$, set $c_\ell = 0$ and $C_\ell = \{\}$.
\item In the case $m_\ell = 1$, if $\ell = p$, set $c_\ell = p^{a - 1}(p+1)$ and \\ $C_\ell = \{(\alpha, 0, pk\alpha, k\alpha) : 0 \le k < p^a\} \cup \{(pk\alpha, 0, p\alpha, \alpha) : 0 \le k < p^{a-1}\}$.
\item In the case $m_\ell = 1$, if $\ell^a$ divides $p-1$, set $c_\ell = \ell^{a - 1}(\ell+1)$ and \\ $C_\ell = \{(\alpha, k\alpha, 0, 0) : 0 \le k < \ell^a\} \cup \{(pk\alpha, \alpha, 0, 0) : 0 \le k < \ell^{a-1}\}$.
\item In the case $m_\ell = 1$, if $\ell^a$ does not divide $p-1$, set $c_\ell = \ggcd(\ell^a, p-1)$ and $C_\ell = \{(\alpha,\frac{\ell^a}{c_\ell} k \alpha, 0, 0) : 0 \le k < c_\ell\}$.
\end{enumerate}

\begin{proposition}
Let $G$ be abelian, with elementary factors as in \eqref{eq:Gdef}.
Then the number of Galois extensions $K/\Qp$ with Galois group $G$ is $\prod_\ell c_\ell$
and the set $\{\sum_\ell \eta_\ell : \eta_\ell \in C_\ell\}$ forms a set of representatives for the orbits of $\Aut(G)$ on $X_G$.
\end{proposition}
\begin{proof}
The role of $x_1$ at $p$ is almost the same as the role of $\tau$ away from $p$, except that the order of $\tau$ must divide $p-1$.  For $\ell \ne p$, the $\ell$-component of $x_1$ must be $0$; the $p$-component of $\tau$ must be $0$.
Therefore, if any $m_\ell$ is at least $3$, it is impossible for $\sigma, \tau$ and $x_1$ to generate $G$.

When $m_\ell = 2$, generating sets for $\ZZ / \ell^a \ZZ \times \ZZ / \ell^b \ZZ$ are permuted transitively by $\Aut(G)$ \cite{hillar-rhea:07a}*{Thm. 3.6}, and if $a = b$ then the two generators can be interchanged by an automorphism.
When $\ell^b$ divides $p-1$ then $\tau$ can be taken as either generator, whereas if $\ell^a$ divides $p-1$ but $\ell^b$ does not then $\tau$ can only be the generator of order $\ell^a$.
If $\ell \ne p$ and $\ell^a$ does not divide $p-1$ then $\sigma$ and $\tau$ cannot generate $G$.

When $m_\ell = 1$ then either $\sigma$ or $\tau$ (or both) must be a generator.  The descriptions of $C_\ell$ then follow from the fact that $\Aut(\ZZ/N\ZZ) \cong (\ZZ/N\ZZ)^\times$.
\end{proof}

\begin{remark}
It is also possible to count abelian extensions using local class field theory, but the orbits on $X_G$ are used in the lifting algorithm of Section \ref{ssec:lifting}.
\end{remark}

\subsection{Tame groups}

If $G$ has order relatively prime to $p$, or more generally if $V$ is trivial, then we must have $x_0 = x_1 = 1$.  We search for elements of $X_G$ by enumerating the normal subgroups that can contain $\tau$,
then finding pairs $(\sigma, \tau)$ that satisfy the tame relation and generate $G$.  We summarize the steps in Algorithm \ref{alg:tamecount}.

\begin{algorithm}
\SetKwData{N}{N}\SetKwData{D}{D}\SetKwData{G}{G}\SetKwData{s}{s}\SetKwData{t}{t}\SetKwData{pairs}{pairs}
\SetKwFunction{DerivedSubgroup}{DerivedSubgroup}
\SetKwFunction{IsCyclic}{IsCyclic}
\SetKwFunction{NormalSubgroupsAbove}{NormalSubgroupsAbove}
\SetKwInOut{Input}{Input}\SetKwInOut{Output}{Output}
\Input{A group \G with trivial $p$-core}
\Output{A list of pairs $(\sigma, \tau)$ representing the $\Aut(G)$-orbits in $X_G$.}
\BlankLine
\D = \DerivedSubgroup{\G}\;
\pairs = []\;
\If{\IsCyclic{\D}}{
\For{\N in \NormalSubgroupsAbove{\D}}{
\If{\IsCyclic{\N} and \IsCyclic{\G/\N}}{
\For{\s in \G that induce $p$th powering on \N}{
\For{\t in \N that generate \G along with \s}{
\If{$(\s, \t)$ not marked}{
Append $(\s, \t)$ to \pairs\;
Mark images of $(\s, \t)$ under $\Aut(G)$\;
}}}}}}
\Return{\pairs}\;
\caption{Enumerating extensions: tame case} \label{alg:tamecount}
\end{algorithm}

\subsection{Lifting homomorphisms} \label{ssec:lifting}

For potentially $p$-realizable groups $G$ that are neither tame nor abelian, we choose a minimal normal subgroup $N \triangleleft G$ (such an $N$ always exists since $G$ is solvable) and set $Q = G/N$.
Inductively, we may assume that we have computed a list $Y_Q$ of representatives for the orbits of $\Aut(Q)$ on $X_Q$.  In particular, if $Q$ is abelian or tame then we may use Section \ref{ssec:ab} or
Algorithm \ref{alg:tamecount}; otherwise we will recursively use the algorithm described in this section.

The idea is to just test all lifts of quadruples $(\sigma, \tau, x_0, x_1) \in Y_Q$ to see if they are valid elements of $X_G$.  There is a subtlety however: there may be automorphisms of $Q$
which are not induced by automorphisms of $G$.  This problem comes in two parts.  First, if $N$ is not a characteristic subgroup then it may not be stabilized by all of $\Aut(G)$, so not all automorphisms descend.
Second, the map $\Stab_{\Aut(G)}(N) \to \Aut(Q)$ is not necessarily surjective, so elements of $X_Q$ that are equivalent under $\Aut(Q)$ may lift to elements that are inequivalent under $\Aut(G)$.

We solve the problem by computing a list of coset representatives for the image of $\Stab_{\Aut(G)}(N) \to \Aut(Q)$.  Then, instead of just lifting elements of $Y_Q$, we lift all translates under these automorphisms.
We summarize this process in Algorithm \ref{alg:lifting}.

\begin{algorithm} \label{alg:liftcount}
\SetKwData{cokreps}{cokreps}\SetKwData{Xreps}{Xreps}
\SetKwInOut{Input}{Input}\SetKwInOut{Output}{Output}
\Input{A potentially $p$-realizable group $G$ and lists of representatives $Y_Q$ for quotients $Q$ of $G$}
\Output{A list $Y_G$ of quadruples $(\sigma, \tau, x_0, x_1)$ representing the $\Aut(G)$-orbits in $X_G$.}
Choose a minimal normal subgroup $N \triangleleft G$\;
Set $Q = G / N$\;
Compute the stabilizer $A$ of $N$ in $\Aut(G)$\;
Compute a list \cokreps of representatives for the cosets of the image of $A$ in $\Aut(Q)$\;
\Xreps = []\;
\ForEach{$(\sigma, \tau, x_0, x_1) \in Y_Q$}{
\ForEach{$\alpha \in$ \cokreps}{
\ForEach{lift $x_1$ of $\alpha(x_0)$ to $G$ that lies in $V$}{
\ForEach{lift $x_0$ of $\alpha(x_1)$ to $G$ that lies in $V$}{
\ForEach{lift $\tau$ of $\alpha(\tau)$ to $G$ with order prime to $p$}{
\ForEach{lift $\sigma$ of $\alpha(\sigma_0)$ with $\tau^\sigma = \tau^p$}{
\If{$(\sigma, \tau, x_0, x_1)$ not marked}{
Mark images of $(\sigma, \tau, x_0, x_1)$ under $\Aut(G)$\;
\If{$\sigma, \tau, x_0, x_1$ generate $G$}{
Append $(\sigma, \tau, x_0, x_1)$ to \Xreps\;
}}}}}}}}
\Return{\Xreps}\;
\caption{Enumerating extensions: lifting method} \label{alg:lifting}
\end{algorithm}

The runtime of Algorithm \ref{alg:lifting} depends on the structure of $G$.
If $N \triangleleft G$ is the minimal normal subgroup used, $C$ is the list of coset representatives in $\Aut(Q)$,
$Y_Q$ is the list of representatives for the quotient $Q$, and $R$ is the time it takes to compute
the wild relation, then the runtime is bounded by $O(\lvert C \rvert \cdot \lvert Y_Q \rvert \cdot \lvert N \rvert^4 R)$.
The actual runtime may be better for some $N$ since we can short circuit some of the loops
if the lifts of $(x_1, x_0, \tau, \sigma)$ do not satisfy the appropriate conditions.

Running Algorithm \ref{alg:lifting} on groups of order up to $2000$ for $p$ up to $13$
required a few weeks of CPU time.  The largest counts found occurred for cyclic groups
such as $C_{1458}: p=3$ ($2916$) and $C_{1210}: p=11$ ($2376$), or for products of
cyclic groups with small non-abelian groups such as $C_{243} \times S_3: p=3$ ($1944$).
For $p=3$, other nonabelian groups had large counts such as 1458G553:
$(C_{27} \rtimes C_{27}) \rtimes C_2$ ($1323$) suggesting that the dominance of
cyclic groups may not last as the order increases.

Figure \ref{fig:YG_trend} shows these counts in aggregate, ignoring the group structure.
Specifically, recall that $Y_G$ is in bijection with the set of Galois extensions of $\Qp$ with Galois group $G$.
Figure \ref{fig:YG_trend} plots the function $f(n)$ that counts the number of potentially realizable $G$
with $\lvert G \rvert \le 2000$ and $\lvert Y_G \rvert \ge n$.  The difference between the first and second
bars in each chart gives the number of groups that are potentially $p$-realizable but not actually $p$-realizable.
We have truncated the charts at $25$ since they have long tails; the previous paragraph gives examples
of $G$ with large $\lvert Y_G \rvert$.

\begin{figure}
\begin{tikzpicture}
\begin{axis}[
scaled ticks=false,
bar width=2pt,
ymin=0,
ymax=17000,
height=2in,
width=2.1in,
axis x line*=bottom,
axis y line*=left,
xlabel={$n$},
]
\addplot [
ybar,
draw=black,
fill=black,
] coordinates {
(0,16411) (1,13073) (2,10104) (3,8974) (4,8373) (5,6827) (6,6247) (7,5894) (8,5842) (9,5171) (10,4896) (11,4867) (12,4837) (13,3963) (14,3944) (15,3935) (16,3926) (17,3508) (18,3457) (19,3372) (20,3372) (21,3040) (22,3015) (23,3006) (24,3006)
};
\end{axis}
\node at (2,3) {$p=3$};
\end{tikzpicture}%
\begin{tikzpicture}
\begin{axis}[
scaled ticks=false,
bar width=2pt,
ymin=0,
ymax=17000,
height=2in,
width=2.1in,
hide y axis=true,
axis x line*=bottom,
xlabel={$n$},
]
\addplot [
ybar,
draw=black,
fill=black,
] coordinates {
(0,7136) (1,6747) (2,4653) (3,3851) (4,3169) (5,2305) (6,2287) (7,1570) (8,1570) (9,1380) (10,1348) (11,1346) (12,1328) (13,1171) (14,1020) (15,1020) (16,1014) (17,982) (18,880) (19,730) (20,730) (21,722) (22,700) (23,686) (24,684)
};
\end{axis}
\node at (2,3) {$p=5$};
\end{tikzpicture}%
\begin{tikzpicture}
\begin{axis}[
scaled ticks=false,
bar width=2pt,
ymin=0,
ymax=17000,
height=2in,
width=2.1in,
hide y axis=true,
axis x line*=bottom,
xlabel={$n$},
]
\addplot [
ybar,
draw=black,
fill=black,
] coordinates {
(0,6178) (1,6012) (2,3956) (3,3239) (4,2413) (5,1843) (6,1843) (7,1658) (8,1657) (9,1065) (10,963) (11,963) (12,963) (13,792) (14,792) (15,792) (16,782) (17,624) (18,624) (19,577) (20,566) (21,566) (22,565) (23,563) (24,510)
};
\end{axis}
\node at (2,3) {$p=7$};
\end{tikzpicture}%
\caption{\label{fig:YG_trend} Number of pot. $p$-realizable $G$ with $\lvert G \rvert \le 2000$ and $\lvert Y_G \rvert \ge n$}
\end{figure}

We do not have theoretical results on the possible sizes of $N$ and $C$, but experimental results are summarized in Tables \ref{tbl:Nsize} and \ref{tbl:coksize}.
The first shows the number of $G$ that have a specified minimum size of $N$, and the second shows the number of pairs $(G, N)$ with a specified size of $C$, which we refer to as the \emph{automorphism index}.

\vspace{0.2in}

\begin{minipage}{\linewidth}
\centering
\captionof{table}{Automorphism index for nonabelian, non-tame $G$} \label{tbl:coksize}
\newcolumntype{R}[1]{>{\hsize=#1\hsize\raggedleft\arraybackslash}X}%
\newcolumntype{C}[1]{>{\hsize=#1\hsize\centering\arraybackslash}X}%
\begin{tabularx}{\textwidth}{R{0.5} | C{1.1} C{1.1} C{1.1} C{1.1} C{1.1} |}
\\[-0.28in]
\cline{2-6}
& \multicolumn{5}{c|}{Number of $N \triangleleft G$ with given automorphism index} \\
\cline{2-6}
Index & $p=3$ & $p=5$ & $p=7$ & $p=11$ & $p=13$ \\
\hline
1 & 8594 & 2393 & 1210 & 561 & 663 \\
2 & 1798 & 594 & 421 & 111 & 117 \\
3 & 468 & 24 & 73 & 25 & 19 \\
4 & 396 & 157 & 59 & 107 & 17 \\
5 & 0 & 7 & 0 & 4 & 0 \\
6 & 333 & 10 & 58 & 0 & 6 \\
8 & 217 & 42 & 47 & 17 & 13 \\
9 & 91 & 0 & 4 & 0 & 0 \\
10 & 2 & 0 & 0 & 0 & 0 \\
12 & 153 & 7 & 4 & 7 & 1 \\
13 & 21 & 0 & 0 & 0 & 0 \\
16 & 37 & 0 & 8 & 0 & 1 \\
18 & 61 & 0 & 2 & 0 & 0 \\
20 & 0 & 4 & 0 & 1 & 0 \\
24 & 99 & 30 & 4 & 12 & 1 \\
$>24$ & 428 & 12 & 7 & 2 & 0 \\
\cline{2-6}
\end{tabularx}
\end{minipage}

\vspace{0.2in}

\begin{minipage}{\linewidth}
\centering
\captionof{table}{Smallest $N \triangleleft G$ for nonabelian, non-tame $G$} \label{tbl:Nsize}
\newcolumntype{R}[1]{>{\hsize=#1\hsize\raggedleft\arraybackslash}X}%
\newcolumntype{C}[1]{>{\hsize=#1\hsize\centering\arraybackslash}X}%
\begin{tabularx}{\textwidth}{R{0.5} | C{1.1} C{1.1} C{1.1} C{1.1} C{1.1} |}
\\[-0.28in]
\cline{2-6}
& \multicolumn{5}{c|}{Number of groups whose $N$ has the given size} \\
\cline{2-6}
Size & $p=3$ & $p=5$ & $p=7$ & $p=11$ & $p=13$ \\
\hline
2 & 8765 & 2437 & 1419 & 638 & 588 \\
3 & 3800 & 423 & 228 & 104 & 110 \\
5 & 27 & 392 & 70 & 26 & 45 \\
7 & 10 & 6 & 168 & 11 & 18 \\
9 & 87 & 0 & 0 & 0 & 0 \\
11 & 0 & 3 & 0 & 56 & 7 \\
13 & 0 & 3 & 0 & 0 & 68 \\
$>13$ & 9 & 17 & 12 & 12 & 2 \\
\cline{2-6}
\end{tabularx}
\end{minipage}

Large indices did occur, but rarely.  There were $20$ cases of index larger than 10000 for $p=3$, the largest being $4586868$. For $p=5$, the only index larger than $124$ was $3100$, occurring $3$ times;
for other $p$ no index larger than $120$ occurred.

\section{The inverse Galois problem for $p$-adic fields} \label{sec:invprob}

\subsection{Examples of non-realizable groups}
Recall that $G$ is \emph{$p$-realizable} if there exists an extension $K/\Qp$ with $\Gal(K/\Qp) \cong G$.
If $G$ is $p$-realizable, then every quotient of $G$ is as well, leading us to consider the following class
of groups.
\begin{definition}
A group $G$ is \emph{minimally unrealizable} if $G$ is not $p$-realizable but it is potentially $p$-realizable and every proper quotient of $G$ is $p$-realizable.
\end{definition}

In Table \ref{tbl:abpcore} we list the minimally unrealizable $G$ that have abelian $p$-core.  The label is from the GAP SmallGroups library, which makes precise the description of the group; we write $\Fp^n$ for $C_p^n$ to emphasize the vector space structure.  The column $V$ describes the decomposition of $V$ into indecomposable submodules: $n^k$ refers to a submodule of dimension $n$ occurring with multiplicity $k$.
The columns SS, TD, and XC will be described in Section \ref{ssec:realizable}.

\vspace{0.3in}

%\begin{minipage}{\linewidth}
{\centering
\captionof{table}{Minimally unrealizable groups with abelian $p$-core} \label{tbl:abpcore}
\vspace{-0.2in}
\begin{longtable}{cr@{G}llcccc}
$p$ & \multicolumn{2}{c}{Label} & Description & $V$ & SS & TD & XC \\ \hline \endhead
$3$ & 27&5 & $\FF_3^3$ & $1^3$ & N & Y & Y\\
$3$ & 36&7 & $\FF_3^2 \rtimes C_4$ & $1^2$ & Y & Y & Y \\
$3$ & 54&14 & $\FF_3^3 \rtimes C_2$ & $1^3$ & Y & N & N \\
$3$ & 72&33 & $\FF_3^2 \rtimes D_8$ & $1^2$ & Y & Y & Y \\
$3$ & 162&16 & $C_9^2 \rtimes C_2$ & $1^2$ & Y & N & N \\
$3$ & 324&164 & $\FF_3^4 \rtimes C_4$ & $2^2$ & Y & N & Y \\
$3$ & 324&169 & $\FF_3^4 \rtimes (C_2 \times C_2)$ & $1^2 \oplus 1^2$ & Y & N & N \\
$3$ & 378&51 & $\FF_3^2 \rtimes (C_7 \rtimes C_6)$  & $1^2$ & Y & Y & Y \\
$3$ & 648&711 & $\FF_3^4 \rtimes C_8$ & $2^2$ & Y & N & Y \\
$5$ & 50&4 & $\FF_5^2 \rtimes C_2$ & $1^2$ & Y & Y & Y \\
$5$ & 125&5 & $\FF_5^3$ & $1^3$ & N & Y & Y \\
$5$ & 200&20 & $\FF_5^2 \rtimes C_8$ & $1^2$ & Y & Y & Y \\
$5$ & 300&34 & $\FF_5^2 \rtimes (C_3 \rtimes C_4)$ & $1^2$ & Y & Y & Y \\
$5$ & 400&149 & $\FF_5^2 \rtimes (C_8 \times C_2)$ & $1^2$ & Y & Y & Y \\
$5$ & 500&48 & $\FF_5^3 \rtimes C_4$ & $1^3$ & Y & N & Y \\
$5$ & 1300&29 & $\FF_5^2 \rtimes (C_{13} \rtimes C_4)$ & $1^2$ & Y & Y & Y \\
$5$ & 1300&30 & $\FF_5^2 \rtimes (C_{13} \rtimes C_4)$ & $1^2$ & Y & Y & Y \\
$5$ & 1875&21 & $\FF_5^4 \rtimes C_3$ & $2^2$ & Y & Y & Y \\
$7$ & 98&4 & $\FF_7^2 \rtimes C_2$ & $1^2$ & Y & Y & Y \\
$7$ & 147&4 & $\FF_7^2 \rtimes C_3$ & $1^2$ & Y & Y & Y \\
$7$ & 343&5 & $\FF_7^3$ & $1^3$ & N & Y & Y \\
$7$ & 588&22 & $\FF_7^2 \rtimes C_{12}$ & $1^2$ & Y & Y & Y \\
$7$ & 882&23 & $\FF_7^2 \rtimes C_{18}$ & $1^2$ & Y & Y & Y \\
$7$ & 1176&130 & $\FF_7^2 \rtimes (C_3 \times D_8)$ & $1^2$ & Y & Y & Y \\
$11$ & 242&4 & $\FF_{11}^2 \rtimes C_2$ & $1^2$ & Y & Y & Y \\
$11$ & 605&4 & $\FF_{11}^2 \rtimes C_5$ & $1^2$ & Y & Y & Y \\
$11$ & 1331&5 & $\FF_{11}^3$ & $1^3$ & N & Y & Y \\
$13$ & 338&4 & $\FF_{13}^2 \rtimes C_2$ & $1^2$ & Y & Y & Y \\
$13$ & 507&4 & $\FF_{13}^2 \rtimes C_3$ & $1^2$ & Y & Y & Y \\
$13$ & 676&10 & $\FF_{13}^2 \rtimes C_4$ & $1^2$ & Y & Y & Y \\
$13$ & 1014&9 & $\FF_{13}^2 \rtimes C_6$ & $1^2$ & Y & Y & Y \\
\end{longtable}}
%\end{minipage}

% Needed since longtable broke across a page.
\addtocounter{table}{-1}

\vspace{0.15in}
 
\begin{minipage}{\linewidth}
\centering
\captionof{table}{Minimally unrealizable groups with nonabelian $p$-core} \label{tbl:nabpcore}
\begin{tabular}{lllll}
$p$ & Label & Description & $G/W$ & $V/W$ \\
\hline
$3$ & 486G146 & $(\FF_3^4 \rtimes C_3) \rtimes C_2$ & 54G13 & $1^2 \oplus 1$ \\
$3$ & 648G218 & $(C_{27} \rtimes C_3) \times D_8$ & 72G37 & $1^2$ \\
$3$ & 648G219 & $(\FF_3^3 \rtimes C_3) \times D_8$ & 72G37 & $1^2$ \\
$3$ & 648G220 & $((C_9 \times C_3) \rtimes C_3) \times D_8$ & 72G37 & $1^2$ \\
$3$ & 648G221 & $((C_9 \times C_3) \rtimes C_3) \times D_8$ & 72G37 & $1^2$ \\
$3$ & 972G816 & $(\FF_3^2 \times (\FF_3^2 \rtimes C_3)) \rtimes (C_2^2)$ & 324G170 & $1^2 \oplus 1 \oplus 1$ \\
$3$ & 1458G613 & $((C_{81} \times C_3) \rtimes C_3) \rtimes C_2$ & 18G4 & $1^2$ \\
$3$ & 1458G640 & $(C_9^2 \rtimes C_9) \rtimes C_2$ & 18G4 & $1^2$ \\
\end{tabular}
\end{minipage}

\subsection{Realizability criteria} \label{ssec:realizable}

We may explain many of the groups in Table \ref{tbl:abpcore} by considering $V/W$
as a representation of $T = G/V$ on an $\Fp$-vector space.
Note that $\lvert T \rvert$ may be divisible by $p$: this will occur precisely when
there is more than one $p$-Sylow subgroup in $G$.
In this case $V/W$ may not have a decomposition as a direct sum of irreducible subrepresentations,
but it still has a decomposition as a direct sum of indecomposable subrepresentations.
The multiplicity of an indecomposable factor is the number of times it appears in such a representation.

Recall from Definition \ref{def:TGXG} that $T_G$ is the set of pairs $(\sigma, \tau) \in G^2$ generating $G/V$
and satisfying the tame relation.
In order to show that a potentially $p$-realizable group $G$ is not $p$-realizable, we will show that
any possible $(\sigma, \tau, x_0, x_1) \in X_G$ that satisfy the tame and wild relations cannot generate $G$.
We will say that $G$ is \emph{strongly split} (SS) if, for every $(\sigma, \tau) \in T_G$,
the order of $\sigma$ in $G$ equals the order of its image in $G/V$.  Note that Conjecture \ref{conj:semidirect}
would imply that there is some $\sigma$ with the same order in $G$ as in $G/V$,
but some lifts of $\sigma$ from $G/V$ to $G$ may have larger order.

We will say that $G$ is \emph{tame-decoupled} (TD) if $\tau$ acts trivially on $V/W$ for every $(\sigma, \tau) \in T_G$.
Finally, we will say that $G$ is \emph{$x_0$-constrained} (XC) if the implication
\[
x_0^\sigma \langle x_0, \tau \rangle^{-1} \in W \Rightarrow x_0 \in W
\]
holds for all $(\sigma, \tau) \in T_G$.  The last three columns of Table \ref{tbl:abpcore} record whether
$G$ is strongly split, tamely-decoupled and $x_0$-constrained, respectively.

\begin{proposition}
If $G$ is tame-decoupled then it is $x_0$-constrained.
\end{proposition}
\begin{proof}
Each condition holds for $G$ if and only if it holds for $G/W$, so we may assume that $V$
is an elementary abelian $p$-group and $W=1$.
Since every $\tau$ acts trivially on $V$ by conjugation and $h$ is a $(p-1)$st root of unity,
\[
\langle x_0, \tau \rangle = (x_0^{1 + h + \dots + h^{p-2}} \tau^{p-1})^{\frac{\pi}{p-1}} = \tau^\pi = 1.
\]
So if $x_0^\sigma \langle x_0, \tau \rangle^{-1} = 1$ then $x_0^\sigma = 1$ and thus $x_0 = 1$.
\end{proof}
Let $n_{G, \sss}$ be $0$ if $G$ is strongly split and $1$ otherwise; let $n_{G, \xc}$ be $0$ if $G$ is $x_0$-constrained
and $1$ otherwise.

\begin{theorem} \label{thm:multiplicity}
Suppose $G$ is potentially $p$-realizable.  Let $n$ be the largest multiplicity of an indecomposable factor of $V/W$
as a representation of $T$.  If $n > 1 + n_{G, \sss} + n_{G, \xc}$, then $G$ is not $p$-realizable.
\end{theorem}
\begin{proof}
We first reduce to the case where $W=1$.
This is easily done, since the definitions of $n$, $n_{G, \sss}$ and $n_{G, \xc}$ are invariant
under quotienting by $W$, and if we can show that $G/W$ is not $p$-realizable then
$G$ will be unrealizable as well.  We may therefore replace $V$ by $V/W$ and assume
that $V$ is an elementary abelian $p$-group.  

For sake of contradiction, suppose that $G$ is $p$-realizable, with $(\sigma, \tau, x_0, x_1) \in X_G$.
Suppose that we have an arbitrary word in these generators, and assume that the word is an element of $V$.
Using the conjugation action of $T$ on $V$ and the tame relation,
we may rewrite it as $\sigma^c \tau^d x$, where $x$ is a product of conjugates of $x_0$ and $x_1$ under the
action of $T$.  Thus $\sigma^c \tau^d \in V$, so we may use the fact that $\tau$ has order prime to $p$ to
rewrite $\sigma^c \tau^d$ as $\sigma^{c'} \in V$.  If $G$ is strongly split then we must have
$\sigma^{c'} = 1$; otherwise it could be some nonzero element of $V$.

Since $V$ is an elementary abelian $p$-group, the wild relation \eqref{eq:wild_rel} simplifies to
\begin{equation} \label{eq:elab}
x_0^\sigma \langle x_0, \tau \rangle^{-1} = 1.
\end{equation}
If $G$ is $x_0$-constrained, we must have $x_0 = 1$; otherwise $x_0$ can be nontrivial.

Since $x_1$ is unconstrained, we can write any word in terms of a fixed set of $1 + n_{G, \sss} + n_{G, \xc}$ elements of $V$,
where we are allowed to act on these elements by $T$.  Let $A$ be a homogeneous component of
$V$ with multiplicity $n$, and consider the projections of our $1 + n_{G, \sss} + n_{G, \xc}$ elements onto $A$.
Their $\Fp[T]$-span is a proper subspace of $A$ since $A$ has multiplicity $n > 1 + n_{G, \sss} + n_{G, \xc}$,
contradicting the assumption that $(\sigma, \tau, x_0, x_1)$ generate $G$.\end{proof}

We can get a partial converse, but we now need to assume that $W=1$.

\begin{theorem} \label{thm:converse}
Suppose that $G$ is potentially $p$-realizable with $W = 1$, and that $V$ decomposes as a
multiplicity-free direct sum of irreducible $T$-submodules.  Then $G$ is $p$-realizable.
\end{theorem}
\begin{proof}
It suffices to construct an element of $X_G$.  Since $V$ is an elementary abelian $p$-group,
we again have the relation \eqref{eq:elab}, which is satisfied for $x_0 = 1$ and arbitrary $x_1$.
Since $G$ is potentially $p$-realizable, by Proposition \ref{prop:Vsuffices} there are
$\sigma, \tau \in G$ satisfying the tame relation and generating $G/V$.  Choose $x_1 \in V$
with nonzero projection onto each irreducible component. The conjugates of $x_1$ under $T$
generate $V$, since if they were contained in a proper subspace that subspace would have zero
projection onto some irreducible component, contradicting the choice of $x_1$.  Now the fact that
$\sigma$ and $\tau$ generate $G/V$ means that $x_1, \sigma$ and $\tau$ generate $G$.
\end{proof}

\begin{remark}
There are two groups in Table \ref{tbl:abpcore} that are not explained by Theorem \ref{thm:multiplicity}.  For 324G169, there are nonzero $x_0$ satisfying \eqref{eq:elab},
but they all lie in a $1$-dimensional indecomposable subrepresentation.  The other subrepresentation can't be spanned by $x_1$ on its own.  For 162G16, the quotient by $W$ is $p$-realizable.
Here $V$ is abelian but has exponent 9 rather than 3, so the wild relation takes the form
\begin{equation}
x_0^\sigma \langle x_0, \tau \rangle^{-1} = x_1^p.
\end{equation}
In order to get a nontrivial $x_1$, we need to find $x_0$ with $x_0^\sigma \langle x_0, \tau \rangle^{-1}$ of order $3$.  Such $x_0$ exist, but they all have the property that $x_0^\sigma \langle x_0, \tau \rangle^{-1}$
is a multiple of $x_0$, preventing $x_1$ from spanning the rest of $V$.
\end{remark}

\begin{remark}
Table \ref{tbl:nabpcore} gives the groups of order up to $2000$ with nonabelian $V$ that are minimally unrealizable.  In each case, $G/W$ will be $p$-realizable, so the methods
of this section do not apply.  In order to provide an explanation for why they are not $p$-realizable, one would need to analyze the wild relation more thoroughly.
\end{remark}

\bibliographystyle{plain}
\bibliography{../bibliography/Biblio}

\end{document}